\documentclass[12pt]{amsart}
\usepackage{epsfig,color}
\usepackage{dsfont}
\usepackage{graphicx}
\usepackage{amsmath}
\usepackage{graphicx}
\usepackage{color}
\usepackage{comment}

\usepackage{amsmath,amssymb,amsthm,epsfig}
\usepackage[utf8]{inputenc}
\usepackage{xcolor}
\usepackage[pdftex,colorlinks]{hyperref}
\usepackage{soul, graphicx,  enumitem}
\usepackage[all]{xy}

\makeatother

\newtheorem{theor}{theorem}[section]
\newtheorem{theorem}[theor]{Theorem}

\newtheorem{corollary}[theor]{Corollary}

\newtheorem{remark}[theor]{Remark}
\newtheorem{lemma}[theor]{Lemma}

\newcommand{\eop}[1]{{\flushright\hfill\fbox{\bf #1}}}

\newtheorem{theo}{theoreme}[section]

\newtheorem{proposition}[theo]{Proposition}

\newcommand{\Ric}{\operatorname{Ric}}
\newcommand{\Scal}{\operatorname{Scal}}

\newcommand{\trace}{\operatorname{trace}}

\title{On Horospherical Rigidity}

\author{G\'erard Besson}
\address{Institut Fourier\\ 
	Universit\'e Grenoble Alpes\\ 
	Institut Fourier\\
	100 rue des maths, 38610 Gi\`eres}
\email{g.besson@univ-grenoble-alpes.fr}

\author{Gilles Courtois}
\address{Department of Mathematics\\Paris VI\\4 place Jussieu, 75232 Paris C\'edex 09}
\email{gilles.courtois@imj-prg.fr}

\author{Sa'ar Hersonsky}
\address{Department of Mathematics\\ 
	University of Georgia\\ 
	Athens, GA 30602}
\email{saarh@uga.edu}

\thanks{}
\date{\today}

\begin{document}
\maketitle

\section{Introduction}

This article concerns the question of how the geometry of the horospheres of a closed negatively curved manifold of dimension greater than $2$ determines the geometry of the whole manifold. 
Relations between the extrinsic geometry of the horospheres and the geometry of $M$ have already been considered. For instance, \cite{FoL} and \cite{BCG1} show that if all the horospheres have constant mean curvature, then the underlying manifold is locally symmetric (of negative curvature). Let us recall that the mean curvature of a hypersurface is related to the derivative of its volume element in the normal direction to the hypersurface, and hence the mean curvature is an extrinsic quantity. In contrast, we give in this work a characterization of the hyperbolic space in term of the an {\sl intrinsic} geometric property of the horospheres.

Before stating our main theorem, let us recall a few important features of the manifolds under consideration and results that are related to our work in this paper.  Let 
$(M, g)$ denote a $n$-dimensional, closed, Riemannian manifold  endowed with a metric of negative sectional curvature and of dimension $n\geq 3$. It follows from the Cartan-Hadamard theorem that  $\tilde M$, the universal cover of $M$ endowed with the metric $\tilde g$, is diffeomorphic to $\mathbf{R}^{n}$. 

Given a point, $\tilde m_0\in \widetilde M$, and a unit tangent vector, $\tilde v\in T_{\tilde m_0}\widetilde M$, we let  $c_{\tilde v}$  denote the unique geodesic ray determined by  $c_{\tilde v}(0)=\tilde m_0$ and $\dot c _{\tilde v}(0) =\tilde v$. It is well known that the map, $\tilde v \in T_{x_0} \widetilde M \mapsto [c_{\tilde  v}] \in \partial \tilde M$, defines a homeomorphism between the unit sphere in $T_{\tilde m_0} \widetilde M$ and $\partial \widetilde M$. Given a point $\xi = [c_{\tilde v}] \in \partial \widetilde M$, the Busemann function $B_\xi(\cdot)$ is then defined for all $\xi \in \partial \widetilde M$ and for all $\tilde m \in \widetilde M$, by $B_{\xi}(\tilde m)= \lim_{t\to \infty} (d(\tilde m,c_{\tilde v} (t)) - d(\tilde m_0, c_{\tilde v} (t)))$. 

Since $M$ is a closed negatively curved manifold, it is known that for each $\xi \in \partial \widetilde M$  the Busemann function $B_\xi(\cdot)$ is $C^\infty$-smooth.  Furthermore, for any $t\in\mathbf{R}$, the level set  $$H_{\xi}(t) =\left\{\tilde m\in {\widetilde M};\, B_\xi(\tilde m)= t\right\}$$ is a smooth submanifold of $\widetilde M$ which is diffeomorphic to $\mathbf{R}^{n}$ and which is  called a {\sl horosphere} centred at $\xi$. It follows that  horospheres inherit a complete Riemannian metric induced by the restriction of the metric of $\widetilde M$. For instance, if $(M,g)$ is a real hyperbolic manifold, every horosphere of $\widetilde M$ is flat and therefore isometric to the Euclidean space $\mathbf R^{n}$.

We defined horospheres as special submanifolds in $\widetilde M$ and let us give now a dynamical perspective. Let $\tilde p : T^1\widetilde M \rightarrow \widetilde M$ and $p : T^1M \rightarrow M$ denote  the natural projections. The geodesic flow on $T^1\widetilde M$ is known to be an Anosov flow, that is, the tangent bundle  $TT^1 \widetilde M$ admits a decomposition as $TT^1\widetilde M = \mathbf R X \oplus \widetilde E^{ss} \oplus \widetilde E^{su}$, where $X$ is the vector field generating the geodesic flow and $\widetilde E^{ss}$, $\widetilde E^{su}$ are the strong stable and strong unstable distributions, respectively. These distributions are known to be integrable, invariant under the differential of the geodesic flow, and to give rise to two transverse foliations of $T^1\widetilde M$, $\widetilde W^{ss}$ and $\widetilde W^{su}$, the strong stable and strong unstable foliations, respectively, whose leaves are smooth submanifolds. A classical property of these foliations is that in general they are transversally H\" older with exponent less than one.  A link between the two point of views on horospheres is the following. For $\tilde v\in T^1\widetilde M$, the strong unstable leaf $\widetilde W^{su}(\tilde v)$ through $\tilde v$ is defined to be  the set of unit vectors $\tilde w\in T^1\widetilde M$ which are normal to the horosphere $H_\xi (0)$ and pointing  outward,  that is, on the same side as $\tilde v$. For the sake of simplicity this horosphere defined by $\tilde v$ will be noted in the sequel by $H(\tilde v)$. Note that the unstable foliation being in general only Hölder  echoes the fact that the Busemann functions are only Hölder continuous in $\xi$; however, for any fixed $\xi$, we get a foliation of $\widetilde M$ by the whole family of horospheres centred at $\xi$ which is $C^\infty$-smooth as well as the Busemann function $B_\xi( .)$ as a function of $\tilde m\in \widetilde M$. More details are given in \cite{BCH}, the Introduction and Section 2.

In sum, when $(M, g)$ is a closed negatively curved manifold the horospheres are smooth and carry the Riemannian metric induced by $\tilde g$. From the previous comments we know that the geometric quantities related to a given horosphere $H(\tilde v)$, such as its curvature and its second fundamental form, are continuous on $T^1\widetilde M$ and smooth along the flow lines. Our assumptions, stated in the Theorem below,  are concerned  with the scalar curvature of the horospheres. We let  $\tilde s(\tilde v)$ denote the scalar curvature of $H(\tilde v)$ at the base point $\tilde p (\tilde v)$ of $\tilde v$. It is a continuous function defined on $T^1\widetilde M$. The natural projection $\pi: \widetilde M\rightarrow M$ maps $H(\tilde v)$ onto an immersed hypersurface in $M$ which is denoted, by abuse of language, by $H(v)$ where $m=\pi (\tilde m)$ and $v= d\pi (\tilde v)\in T^1_m(M)$. By invariance by the deck transformations, the scalar curvature $\tilde s (\cdot)$ descends to a smooth function defined on $T^1(M)$, and for $v\in T^1M$ we denote it by $s(v)$. We also let  $\mu_L$ be the Liouville measure on $T^1M$. Our main theorem is the following.
\begin{theorem}\label{main}
	Let $(M^n, g)$ be a closed Riemannian manifold of dimension $n\geq 3$ and of negative curvature. Let us assume that 
	\begin{itemize}
		\item[i)] either $\int_{T^1M}s(v)d\mu_L\geq 0$, or
		\item[ii)] there exists one horosphere, $H(v_0)$ for some $v_0\in T^1_{m_0}M$, such that $s(w)\geq 0$ for all $w\in T^1_mM$ with $m\in H(v_0)$ and $w$ orthogonal to $H(v_0)$ at $m$.
	\end{itemize}
	Then $(M, g)$ has constant negative sectional curvature.
\end{theorem}

\begin{corollary}
	Let $(M^n, g)$ be a closed Riemannian manifold of dimension $n\geq 3$ and of negative curvature. Let us assume that
	one horosphere, $H(\tilde v_0)$ for some $\tilde v_0\in T^1_{\tilde m_0}\widetilde M$, is flat for the Riemannian metric induced by $\tilde g$. Then $(M, g)$ has constant negative sectional curvature.
\end{corollary}

In \cite{BCH}, we recently proved a rigidity result of a dynamical nature. The main result proved in this paper  is geometric and complements the study started while also opens a series of research questions. The main ingredient in the proof of Theorem \ref{main} is the Riccati equation satisfied by the shape operator of the horospheres. The use of Riccati equation in the study of global geometric properties of Riemannian manifolds is not new, cf. \cite{Green}, \cite{Knieper}, \cite{IS}, \cite{ISS}. For example, it is central in L. W. Green's proof that Riemannian manifolds without conjugate points have non positive total scalar curvature and zero total scalar curvature if and only if the metric is Euclidean.  More recently, M.~Itoh and H.~Satoh characterized the real space forms, i.e., the Cartan Hadamard manifolds of constant non positive sectional curvature, by the property that the horospheres are totally umbilic, i.e., have constant principal curvatures, \cite{IS}. In this latter work, the Riccati equation also plays an essential role. In \cite{ISS}, the authors also give various characterizations of the complex and quaternionic hyperbolic spaces. It is important to note that  \cite{Knieper}, \cite{IS} and \cite{ISS} provide characterizations of hyperbolic real space in term of the {\sl extrinsic} geometry of horospheres.

\section{Proof of the main theorem}

We recall that given $\tilde m\in \widetilde M$ and $\tilde v\in T^1_{\tilde m}\widetilde M$ there is a unique horosphere $H(\tilde v)\subset \widetilde M$, orthogonal to $\tilde v$ at $\tilde m$. When $(M, g)$ is a closed negatively curved manifold the horospheres are smooth and carry the Riemannian metric induced by $\tilde g$. Our assumptions are concerned with its scalar curvature, denoted by $s(\tilde v)$, which is a continuous function defined on $T^1\widetilde M$. The natural projection $\pi: \widetilde M\rightarrow M$ maps $H(\tilde v)$ to an immersed hypersurface in $M$.  By abuse of notation, it  will be denoted by $H(v)$,  where $m=\pi (\tilde m)$ and $v= d\pi (\tilde v)\in T^1_m(M)$. By invariance by the deck transformations, the scalar curvature descends to a continuous  function defined on $T^1M$, and for $v\in T^1M$ we denote it by $s(v)$. We also denote by $\mu_L$ the Liouville measure on $T^1M$. We now turn to the proof of our Main Theorem.

\smallskip
{\bf {\em Proof of Theorem~\ref{main}}.}

\smallskip
Assertion $ii)$ implies that the scalar curvature function, $s(\cdot)$, is non negative on the lift to $T^1(M)$ of one horosphere. We first note that since $(M, g)$ is {\sl negatively curved} $s(\cdot)$ is then non negative everywhere in $T^1M$. More precisely,

\begin{proposition}\label{pro:one-all}
Let $(M^n, g)$ be a closed Riemannian manifold of dimension $n\geq 3$ and of negative sectional curvature. If one horosphere has non negative scalar curvature then every horosphere has non negative scalar curvature. In particular assertion $ii)$ implies assertion  $i)$.
\end{proposition}

The proof of this proposition relies on the density of each horosphere viewed as a leaf of the unstable foliation in the unit tangent bundle $T^1M$. The proof follows exactly the same lines as the one given for the second assertion of  Proposition 2.1 in \cite{BCH}. 

\smallskip

We now proceed to the proof of the main theorem. To $\tilde m \in  H(\tilde v)$  for $\tilde v\in T_{\tilde m}\widetilde M$ we associate $S:=S_{\tilde v}$ the shape operator of $H(\tilde v)$ as follows:  let $\tilde u\in T^1_{\tilde m}H(\tilde v)$, then
$$S: \tilde u\to\nabla_{\tilde u}\tilde N\in T^1_{\tilde m}H(\tilde v) ,$$
where $\tilde N$ is the unit normal vector field of $H(\tilde v)$ (pointing outward). Note that some authors choose another possible convention for the shape operator, this one is chosen so that it is positive on a standard Euclidean sphere. Since the metric on $\tilde M$ is negatively curved, the Busemann functions are convex, hence, $S$ is a non negative symmetric operator acting on the tangent space to this horosphere at $\tilde m$. The shape operator  is a $(1,1)$-tensor (see \cite{Eschenburg}) and we can also view it as the second fundamental form of $H(\tilde v)$ at $\tilde v$, which is a $(2,0)$-tensor.  Henceforth, we will make the following abuse of language
$$\forall \tilde u, \tilde w\in T_{\tilde m}H(\tilde v),\quad \langle S(\tilde u), \tilde w \rangle = S(\tilde u, \tilde w)\,. $$
The shape operator is also defined by using Jacobi Fields (see \cite{Knieper} page 372, where it is called $U^+$ or \cite{Eschenburg} page 4, where it is called $U$). Let c(t) denote the (infinite) geodesic tangent to $\tilde v=\dot c(0)$ at $\tilde m=c(0)$, which is therefore orthogonal to $H(\tilde v)$. It follows  that $S$ is smooth along the geodesic $c(t)$ and it satisfies the following Riccati equation 
\begin{equation}\label{eq:Riccati}
\dot S+S^2+ R_{\dot c}=0\,,
\end{equation}
which is derived from the Jacobi equation. In this equation, $\dot S$ denotes  the covariant derivative, $\nabla_{\dot c}S$, of $S$ along $c(t)$ and $R_{\dot c}$ is the curvature operator defined, for $\tilde u, \tilde w\in T_{\tilde m}H(\tilde v)$, by 
\begin{equation}\label{curv-op}
(\tilde u, \tilde w)\to R_{\dot c}(\tilde{u},\tilde{w})= \langle R(\tilde u, \dot c)\dot c, \tilde w\rangle,
\end{equation}
where $R$ is the curvature tensor. Notice that 
if $\tilde u=\tilde w$, the expression in (\ref{curv-op}) is the sectional curvature of the $2$-plane generated by $\dot c$ and $\tilde u$. Taking the trace of these symmetric operators we get functions depending on $\tilde v$ and hence defined on $T^1\widetilde M$. Furthermore, by invariance we can consider them as functions defined on $T^1M$ and which depend on $v=d\pi (\tilde v)$.

it is known that  the following equation holds (see, for example, \cite{Knieper})
\begin{equation}\label{eq:Riccati2}
X.{\trace}(S)+{\trace}(S^2)+\Ric(v)=0.
\end{equation}
Here, $\Ric (v)$ is the Ricci curvature considered as a quadratic form applied to $v$ and $X$ is the geodesic flow vector field on $T^1M$. Note that in  (\ref{eq:Riccati}) the derivative of $S$ uses the connection of the Riemannian metric of $\widetilde M$ whereas in (\ref{eq:Riccati2}) the derivative of the function $\trace(S)$ is taken along the geodesic $c(t)$. 
Recall that $d\mu_L$, the Liouville measure on $T^1M$, is invariant by the geodesic flow, hence
\begin{equation}
\label{eq:int}
\int_{T^1M}(X.\trace (S))d\mu_L=0.
\end{equation}
\label{eq:leftwith}
After integration we are left with,
\begin{equation}
\int_{T^1M}\trace (S^2)(v)d\mu_L(v)+\int_{T^1M}\Ric (v)d\mu_L(v)=0.
\end{equation}
The second term can be integrated using Fubini's decomposition, the integral of the quadratic form $\Ric$ on the unit sphere at a point is equal to its trace, $\Scal$     (the scalar curvature of $M$) divided by $n$, the dimension of $M$. Hence, with $p: T^1M\to M$ denoting the natural projection we have that,
\begin{equation}\label{eq:Integrated}
\int_{T^1M}\trace (S^2)(v)d\mu_L(v)=-\frac{1}{n}\int_{T^1M}\Scal \circ p(v) d\mu_L(v).
\end{equation}
Finally, The Gauss equation for $H(\tilde v)$ reads (see \cite{Eschenburg} page 7, \cite{Knieper} page 376),
\begin{equation}\label{eq:Gauss}
\trace (S(\tilde v))^2-\trace (S^2(\tilde v))=2\Ric (\tilde v) - \Scal\circ p(\tilde v)+ s(\tilde v),
\end{equation}

Let us denote by $(\lambda_1,\dots , \lambda_{n-1})$, the principal curvatures of $H(\tilde v)$, they are functions defined at every point of $H(\tilde v)$ and for the sake of simplicity we omit to mention it. The computations below are done at every point of the horosphere. The Gauss equation (\ref{eq:Gauss}) becomes,
$$\big(\sum_1^{n-1}\lambda_i\big)^2-\sum_1^{n-1}\lambda_i^2=\sum_{i\ne j}\lambda_i\lambda_j=2\Ric-\Scal\circ p+s.$$
On the other hand we have,
\begin{lemma}\label{lem:SumProduct}
    $$\sum_1^{n-1}\lambda_i^2\geq\frac{1}{n-2}\sum_{j\ne i}\lambda_i\lambda_j.$$
\end{lemma}  
\begin{proof}[Proof of Lemma \ref{lem:SumProduct}]
$$(n-2)\sum_1^{n-1}\lambda_i^2=\sum_{j>i}(\lambda_i^2+\lambda_j^2)\geq \sum_{j>i}(2\lambda_i\lambda_j)=\sum_{i\ne j}\lambda_i\lambda_j.$$
Indeed, in the term $\sum_{j>i}(\lambda_i^2+\lambda_j^2)$ a given index $i$ appears once for each $j>i$ but also once for each $j<i$ by considering $(\lambda_i^2+\lambda_j^2)$, hence $(n-2)$ times.\\
\end{proof}
As a consequence of Lemma \ref{lem:SumProduct} we get,
\begin{equation}\label{eq:SumProduct}
  \trace(S^2)=\sum_1^{n-1}\lambda_i^2\geq\frac{1}{n-2}\sum_{i\ne j}\lambda_i\lambda_j=\frac{1}{n-2}\big(2\Ric-\Scal\circ p+s\big).
\end{equation}
Using any of assumptions of the theorem, we have that $\int_{T^1M}s(v)d\mu_L\geq 0$, and plugging inequality (\ref{eq:SumProduct}) in (\ref{eq:Integrated}) gives that, 

\begin{equation}\label{eq:long}
\begin{split}
	-\frac{1}{n}\int_{T^1M}\Scal \circ p (v) d\mu_L(v)\geq &\frac{1}{n-2}(\frac{2}{n}-1)\int_{T^1M}\Scal \circ p (v) d\mu_L(v)\\
	=& -\frac{1}{n}\int_{T^1M}\Scal \circ p (v) d\mu_L(v).
\end{split}
\end{equation}

Hence, one has equality in (\ref{eq:SumProduct}) which is a superposition of inequalities of the type 
$\lambda_i^2+\lambda_j^2\geq 2\lambda_i\lambda_j$ for $i\ne j$
and then for any couple $(\lambda_i, \lambda_j)$ with $i\ne j$ we have the equality $\lambda_i=\lambda_j$, which shows that every horosphere is umbilical at every point. Notice that, since the shape operators of the horospheres are non negative, all the $\lambda_i$'s are non negative. 

The conclusion of the previous arguments is that under our assumptions $S$ is a multiple of the identity at every point with a coefficient which a priori depends on the vector $\tilde v$. Then the Riccati equation (\ref{eq:Riccati}) shows that $R_{\dot c}$ is also a multiple of the identity. Indeed, we have that,
\begin{lemma}\label{lem:deriv}
  Let us assume that $S=a(\tilde v).id$ then $\nabla_{\dot c}S= da (\dot c).id$
\end{lemma}
\begin{proof}[Proof of Lemma \ref{lem:deriv}]
Let $\tilde u, \tilde w \in T_{p(\tilde v)}H(\tilde v)$. We extend these vectors by parallel transport into vector fields along $c$ denoted by $U(t)$ and $W(t)$. They remain orthogonal to $\dot c(t)$ hence tangent to the horosphere $H(\dot c(t))$. We then compute the derivative in $t$ of $S(U, W)=a(c(t)) \langle U(t), W(t)\rangle$ and the result follows.
\end{proof}
This shows that $\langle R(\tilde u, \dot c)\dot c,  \tilde u\rangle$ does not depend on the (unit) vector $\tilde u\in T^1_mH(\tilde v)$, that is a vector orthogonal to $\tilde v=\dot c(0)$. The sectional curvature of $\widetilde M$ at $\tilde m$ is thus the same on every $2$-plane $(\tilde u, \tilde v)$. It remains to compute its values on the $2$-planes in $T_{\tilde m}H(\tilde v)$ (orthogonal to $\tilde v$) of the type $(\tilde u, \tilde w)$ with $\tilde w\in T^1_{\tilde m}H(\tilde v)$ and orthogonal to $\tilde u$ (and $\tilde v$ by construction). We can then work similarly with the horosphere $H(\tilde u)$ to which $\tilde v$ and $\tilde u$ are tangent and conclude that the sectional curvature takes the same value on the $2$-planes $(\tilde u, \tilde v)$ and $(\tilde u, \tilde w)$. this shows that the sectional curvature is constant (on every tangent $2$-planes) at any $\tilde m\in \widetilde M$. But Schur's Lemma (see \cite{Petersen}, Lemma 2.4 p. 32) asserts that the sectional curvatures is then constant on the whole manifold since $n\geq 3$.
\eop{Theorem~\ref{main}}

\smallskip

\begin{remark}
The above proof will still work if  $(M,g)$ is {\sl non negatively} curved under the assumption that curvature of the horospheres  can be defined. This requires these submanifolds to be $C^3$-smooth so that the induced metric on their tangent space has second derivatives. Unfortunately, at the moment it seems only known that the horosphere in this more general situation are $C^2$. Negative curvature insures that they are smooth.
\end{remark}

\end{document}